\documentclass[11pt,reqno]{amsart}
\usepackage{amsmath}
\usepackage{cases}
\usepackage{mathrsfs}
\usepackage{bbm}
\usepackage{amssymb}
\usepackage{amscd}
\usepackage{amsfonts,latexsym,amsmath,amsthm,amsxtra,mathdots,amssymb,latexsym,mathabx}
\usepackage[all,cmtip]{xy}
\RequirePackage{amsmath} \RequirePackage{amssymb}
\usepackage{color}
\usepackage{colordvi}
\usepackage{multicol}
\usepackage{hyperref}
\usepackage{mathtools}
\usepackage[margin=1in]{geometry}
\usepackage{xcolor}
\hypersetup{
	colorlinks,
	linkcolor={red!50!black},
	citecolor={blue!50!black},
	urlcolor={blue!80!black}
}
\usepackage{cite}

\newcommand{\bea}{\begin{eqnarray}}
	\newcommand{\eea}{\end{eqnarray}}
\newcommand{\bna}{\begin{eqnarray*}}
	\newcommand{\ena}{\end{eqnarray*}}

\numberwithin{equation}{section}

\theoremstyle{plain}
\newtheorem*{Theorem A}{Theorem A}
\newtheorem*{Theorem B}{Theorem B}

\newtheorem{lemma}{Lemma}[section]
\newtheorem{theorem}{Theorem}[section]

\theoremstyle{definition}

\begin{document}
	
	\title
	[{Additive and subtractive bases of $ \mathbb{Z}_m$ in average}]
	{Additive and subtractive bases of $ \mathbb{Z}_m$ in average}
	
	\author
	[G. Liang \quad $\&$ \quad Y. Zhang \quad $\&$ \quad  H. Zuo] {Guangping Liang, \quad Yu Zhang \quad \text{and} \quad Haode Zuo$^\dagger$}
	
	\address{(Guangping Liang) School of Mathematical Science,  Yangzhou University, Yangzhou 225002, People's Republic of China}
	\email{\tt 15524259050@163.com}
	
	\address{(Yu Zhang) School of Mathematics,  Shandong University, Jinan 250100, People's Republic of China}
	\email{\tt yuzhang0615@mail.sdu.edu.cn}
	
	\address{(Haode Zuo) School of Mathematical Science,  Yangzhou University, Yangzhou 225002, People's Republic of China}
	\email{\tt yzzxzhd@yzu.edu.cn}

	

\thanks{$^\dagger$ Corresponding author}	
	
	\begin{abstract}
		Given a positive integer $m$, let $\mathbb{Z}_m$ be the set of residue classes mod $m$. For $A\subseteq \mathbb{Z}_m$ and $n\in \mathbb{Z}_m$, let $\sigma_A(n)$ be the number of solutions to the equation $n=x+y$ with $x,y\in A$. Let $\mathcal{H}_m$ be the set of subsets $A\subseteq \mathbb{Z}_m$ such that $\sigma_A(n)\geq1$ for all $n\in \mathbb{Z}_m$. Let
		$$
		\ell_m=\min\limits_{A\in \mathcal{H}_m}\left\lbrace m^{-1}\sum_{n\in \mathbb{Z}_m}\sigma_A(n)\right\rbrace.
		$$
Following a prior result of Ding and Zhao on Ruzsa's number, we know that
$$
\limsup_{m\rightarrow\infty}\ell_m\le 192.
$$
Ding and Zhao then asked possible improvements on this value.		
In this paper, we prove
$$
\limsup\limits_{m\rightarrow\infty}\ell_m\leq 144.
$$
Moreover, parallel results on subtractive bases of $ \mathbb{Z}_m$ were also investigated here.
		
	\end{abstract}
	
	\maketitle
	\section{Introduction}
	
Let $\mathbb{N}$ be the set of natural numbers and $A$ a subset of $\mathbb{N}$. A remarkable conjecture of Erd\H os and Tur\'{a}n \cite{erdos-turan} states that if all sufficiently large numbers $n$ can be written as the sum of two elements of $A$, then the number of representations of $n$ as the sum of two elements of $A$ cannot be bounded. Progress on this conjecture was made by Grekos, Haddad, Helou and Pihko \cite{GHHP}, who proved that the number of representations cannot be bounded by $5$ which was later improved to $7$ by Borwein, Choi and Chu \cite{BCC}. For more related works on the Erd\H os--Tur\'{a}n conjecture, one can refer to the excellent books of Halberstam and Roth \cite{HRo}, Tao and Vu \cite{TVu}.

A set $A$ is called an {\it asymptotic basis} of natural numbers if all sufficiently large numbers $n$ can be written as the sum of two elements of $A$. Motivated by Erd\H os' question, Ruzsa \cite{Ruzsa} constructed an asymptotic basis $A$ of natural numbers which has a bounded square mean value. Furthermore, Ruzsa considered a variant on the Erd\H os--Tur\'{a}n conjecture. Let $\mathbb{Z}_m$ be the set of residue classes mod $m$ and $A$ a subset of $\mathbb{Z}_m$. For any $n\in \mathbb{Z}_m$, let
$$
\sigma_A(n)=\#\big\{(x,y):n=x+y,~x,y\in \mathbb{Z}_m\big\}.
$$
The Ruzsa number $R_m$ is defined to be the least positive integer $r$ so that there exists a set $A\subseteq \mathbb{Z}_m$ with
$
1\le \sigma_A(n)\le r ~(\forall n\in \mathbb{Z}_m).
$
In his argument, Ruzsa proved that there is an absolute constant $C$ such that $R_m\le C$ for all positive integers $m$. Employing Ruzsa's ideas, Tang and Chen \cite{Tang1} proved that $R_m\le 768$ for all sufficiently large $m$. Later, they \cite{Tang2} obtained that $R_m\le 5120$ for all positive integers $m$. In \cite{Chen}, Chen proved that $R_m\le 288$ for all positive integers $m$, which was recently improved to $R_m\le 192$ by Ding and Zhao \cite{Ding}. On the other hand, S\'{a}ndor and Yang \cite{Sandor} showed that $R_m\ge 6$ for all $m\ge 36$.

Along this line, Ding and Zhao \cite{Ding} asked an average version on Ruzsa's number. Precisely, Let $\mathcal{H}_m$ be the set of subsets $A\subseteq \mathbb{Z}_m$ such that $\sigma_A(n)\geq1$ for all $n\in \mathbb{Z}_m$. Ding and Zhao defined the following minimal mean value amount
$$
\ell_m=\min\limits_{A\in \mathcal{H}_m}\left\lbrace m^{-1}\sum_{n\in \mathbb{Z}_m}\sigma_A(n)\right\rbrace.
$$
As pointed out by Ding and Zhao, their result on $R_m\le 192$ clearly implies
\begin{align}\label{eq1-1}
\limsup_{m\rightarrow\infty}\ell_m\le 192.
\end{align}
Ding and Zhao \cite[Section 3]{Ding} thought that `{\it any improvement of the above bound \rm{(\ref{eq1-1})} would be of interest}'. In this note, we shall make some progress on the improvements of (\ref{eq1-1}).

\begin{theorem}\label{thm1}
We have
$$
\limsup\limits_{m\rightarrow\infty}\ell_m\leq 144.
$$
\end{theorem}

Parallel to the additive bases of $\mathbb{Z}_m$, one naturally considers the corresponding results on subtractive bases of $\mathbb{Z}_m$. Let $A$ be a subset of $\mathbb{Z}_m$. For any $n\in \mathbb{Z}_m$, let
$$
\delta_A(n)=\#\big\{(x,y):n=x-y,~x,y\in \mathbb{Z}_m\big\}.
$$
In \cite{ChenSun}, Chen and Sun proved that for any positive integer $m$, there exists a subset $A$ of $\mathbb{Z}_m$ so that $\delta_A(n)\ge 1$ for any $n\in \mathbb{Z}_m$ and $\delta_A(n)\le 7$ for all $n\in \mathbb{Z}_m$ with $3$ exceptions. Their result was recently improved by Zhang \cite{Zhang} who showed that $\delta_A(n)\le 7$ could be refined to $\delta_A(n)\le 5$, again with $3$ exceptions. The exceptions cannot be removed by their method. Motivated by the minimal mean value amount defined by Ding and Zhao, we consider a parallel amount
$$
g_m:=\min\limits_{A\in \mathcal{K}_m}\left\lbrace m^{-1}\sum_{n\in \mathbb{Z}_m}\delta_A(n)\right\rbrace,
$$
where $\mathcal{K}_m$ is the set of subsets $A\subseteq \mathbb{Z}_m$ such that $\delta_A(n)\geq1$ for all $n\in \mathbb{Z}_m$.
Obviously, Zhang's bound implies that
$$
\limsup\limits_{m\rightarrow\infty}g_m\leq 5
$$
since the total sums of $\delta_A(n)$ on $3$ exceptions contribute only $O(\sqrt{m})$.
Our second main result improves upon this bound slightly.

\begin{theorem}\label{thm2}
We have
$$
\limsup\limits_{m\rightarrow\infty}g_m\leq 2.
$$
\end{theorem}

There is an old conjecture known as the {\it prime power conjecture} (see e.g. \cite{Evans,Guy, Hall}) which states that if $A$ is a subset of $\mathbb{Z}_m$ with $\delta_A(n)=1$ for any nonzero $ n\in \mathbb{Z}_m$, then $m=p^{2\alpha}+p^\alpha+1$, where $p^\alpha$ is a prime power. The reverse direction was proved by Singer \cite{Singer} as early as 1938.

As mentioned by Ding and Zhao \cite{Ding}, it is clear that $\liminf_{m\rightarrow\infty}\ell_m\ge 2$ from \cite[Lemma 2.2]{Sandor}. They \cite[Conjecture 3.3]{Ding} believed that $\liminf_{m\rightarrow\infty}\ell_m\ge 3$. Based on the results of Singer and Theorem \ref{thm2}, it seems reasonable to {\it conjecture}
$$
\lim\limits_{m\rightarrow\infty}g_m=1.
$$
These if true, reflect rather different features between the additive bases and subtractive bases.

\section{Proof of Theorem \ref{thm1}}
For any integer $k$, let
$$
Q_k=\left\{\big(u,ku^2\big):u\in \mathbb{Z}_{p}\right\}\subset \mathbb{Z}_{p}^2.
$$

We will make use of a few lemmas listed below. The first one is quoted from \cite[Lemma 2]{Chen}.

\begin{lemma}[Chen]\label{lem1}
Let $p$ be an odd prime and $m$ a quadratic nonresidue of $p$ with $m+1\not\equiv 0 \pmod{p}, 3m+1\not\equiv 0\pmod{p}$ and $m+3\not\equiv 0 \pmod{p}$. Put
$$
B=Q_{m+1}\cup Q_{m(m+1)}\cup Q_{2m}.
$$
Then for any $(c,d)\in \mathbb{Z}_{p}^2$ we have $1\le \sigma_B(c,d)\le 16$, where $\sigma_B(c,d)$ is the number of solutions of the equation $(c,d)=x+y,~x,y\in B$.
\end{lemma}	

The following lemma is known as the prime number theorem, see e.g. \cite{Davenport}.
\begin{lemma}\label{lem2}
Let $\pi(x)$ be the number of primes $p$ not exceeding $x$. Then
$$
\pi(x)\sim x/\log x, \quad \text{as~} x\rightarrow\infty.
$$
\end{lemma}

The third lemma is simple but important for proofs of our theorem.
\begin{lemma}\label{lem3}
Let $m$ be a positive integer and $A$ a subset of $\mathbb{Z}_m$. Then
$$
\sum_{n\in \mathbb{Z}_m}\sigma_A(n)=|A|^2,
$$
where $|A|$ denotes the number of elements of $A$.
\end{lemma}
\begin{proof}
Clearly, we have
$$
\sum\limits_{n\in \mathbb{Z}_m}\sigma_A(n)=\sum\limits_{n\in\mathbb{Z}_m}\sum\limits_{\substack{a_1+a_2=n\\a_1,a_2\in A}}1=\sum\limits_{\substack{a_1,a_2\in A\\a_1+a_2\in \mathbb{Z}_m}}1=\sum\limits_{a_1,a_2\in A}1=|A|^{2}.
$$
This completes the proof of Lemma \ref{lem3}.
\end{proof}

The fourth lemma deals with a particular case $\mathbb{Z}_{2p^2}$ related to the theme.
\begin{lemma}\label{lem4}
Let $p$ be a prime greater than $11$. Then there is a subset $A\subset \mathbb{Z}_{2p^2}$ with $|A|\le 12p$ so that $\sigma_A(n)\ge 1$ for any $n\in \mathbb{Z}_{2p^2}$.
\end{lemma}
\begin{proof}
Let $p$ be a prime greater than $11$. Then there are at least
$(p-1)/2>5$ quadratic nonresidues mod $p$, which means that there is some quadratic nonresidue $m$ so that
$$
m+1\not\equiv 0 \pmod{p}, \quad 3m+1\not\equiv 0\pmod{p} \quad \text{and} \quad  m+3\not\equiv 0 \pmod{p}.
$$
Let
$
B=Q_{m+1}\cup Q_{m(m+1)}\cup Q_{2m},
$
$$
A_1=\big\lbrace u+2pv:(u,v)\in B\big\rbrace \quad \text{and} \quad  A=A_1\cup(A_1+p),
$$
where $A_1+p:=\{a_1+p:a_1\in A_1\}$. Obviously, $A$ can be viewed as a subset of $\mathbb{Z}_{2p^2}$.

We firstly show that $\sigma_A(n)\ge 1$ for any $n\in \mathbb{Z}_{2p^2}$, i.e., $A\in \mathcal{H}_{2p^2}$ by the definition of $\mathcal{H}_{m}$. We follow the proof of Chen \cite[Theorem 1]{Chen}. For any $(u,v)\in B$, we appoint that $0\le u,v\le p-1$. Let $n$ be an element of $\mathbb{Z}_{2p^2}$ with $0\le n\le 2p^2-1$. Then we can assume that
$$
n=c+2pd
$$
with $p\le c\le 3p-1$ and $-1\le d\le p-1$. By Lemma \ref{lem1}, there are $(u_1,v_1),(u_2,v_2)\in B$ so that
$$
(c,d)=(u_1,v_1)+(u_2,v_2) \pmod{p}
$$
or in other words,
$$
c\equiv u_1+u_2 \pmod{p} \quad \text{and} \quad d\equiv v_1+v_2 \pmod{p}.
$$
Suppose that
$$
c=u_1+u_2+ps \quad \text{and} \quad d=v_1+v_2+ph,
$$
with $s,h\in \mathbb{Z}$, then $s=0$ or $1$ or $2$ since $0\le u_1+u_2\le 2p-2$ and $p\le c\le 3p-1$. Hence, we get
\begin{align*}
n&=c+2pd\\
&= u_1+2pv_1+u_2+2pv_2+ps+2p^2h \\
&\equiv u_1+2pv_1+u_2+2pv_2+ps \pmod{2p^2}.
\end{align*}
If $s=0$, then in $\mathbb{Z}_{2p^2}$ we have
$$
n= (u_1+2pv_1)+(u_2+2pv_2)\in A_1+A_1\subset A+A.
$$
If $s=1$, then in $\mathbb{Z}_{2p^2}$ we have
$$
n= (u_1+2pv_1+p)+(u_2+2pv_2)\in (A_1+p)+A_1\subset A+A.
$$
If $s=2$, then in $\mathbb{Z}_{2p^2}$ we have
$$
n= (u_1+2pv_1+p)+(u_2+2pv_2+p)\in (A_1+p)+(A_1+p)\subset A+A.
$$
Hence, in any case we have $\sigma_A(n)\ge 1$ for $n\in \mathbb{Z}_{2p^2}$.

It can be easily seen that $|A_1|\le 2|B|$
from the construction.
Therefore, for the set $A$ constructed above we clearly have
$$
|A|\le |A_1|+|A_1+p|=2|A_1|\le 2\times 2|B|=4|B|
$$
and
$$
|B|\leqslant|Q_{m+1}|+|Q_{m(m+1)}|+|Q_{2m}|=3p,
$$
from which it follows that
$$
|A|\le 12p.
$$
This completes the proof of Lemma \ref{lem4}.
\end{proof}

The last lemma is a relation between the bases of $\mathbb{Z}_{m_1}$ and $\mathbb{Z}_{m_2}$ with certain constraints.
\begin{lemma}\label{lem5}
Let $\varepsilon>0$ be an arbitrarily small number. Let $m_1$ and $m_2$ be two positive integers with $(2-\varepsilon)m_1<m_2<2m_1$. Suppose that $A$ is a subset of $\mathbb{Z}_{m_1}$ with $\sigma_A(n)\ge 1$ for any $n\in \mathbb{Z}_{m_1}$, then there is a subset $B$ of $\mathbb{Z}_{m_2}$ with $|B|\le 2|A|$ such that $\sigma_B(n)\ge 1$ for any $n\in \mathbb{Z}_{m_2}$.
\end{lemma}	
\begin{proof}
Suppose that $m_2=m_1+r$, then $(1-\varepsilon)m_1<r< m_1$. Let
$$
B=A\cup \{a+r:a\in A\}.
$$
Then we have $|B|\le 2|A|$. It remains to prove $\sigma_B(n)\ge 1$ for any $n\in \mathbb{Z}_{m_2}$.

Without loss of generality, we may assume $0\le a\le m_1-1$ for any $a\in A$.
For $0\le n\le m_1-1$, there are two integers
$a_1,a_2\in A$ so that
$
n\equiv a_1+a_2\pmod{m_1}.
$
Since $0\le a_1+a_2\le 2m_1-2$, it follows that
$$
n=a_1+a_2 \quad \text{or} \quad n=a_1+a_2-m_1.
$$
If $n=a_1+a_2$, then clearly we have $n\equiv a_1+a_2\pmod{m_2}$. If $n=a_1+a_2-m_1$, then
$$
n+m_2=n+m_1+r=a_1+(a_2+r),
$$
which means that $n\equiv a_1+(a_2+r)\pmod{m_2}$.  In both cases we have
$\sigma_B(n)\ge 1$ for any $0\le n\le m_1-1$. We are left over to consider the case $m_1\le n\le m_2-1$. In this range, we have
$$
0<n-r\le m_2-1-r= m_1-1.
$$
Thus, there are two elements $\widetilde{a_1},\widetilde{a_2}$ of $A$ so that
$$
n-r\equiv \widetilde{a_1}+\widetilde{a_2}\pmod{m_1}.
$$
Again by the constraint $0\le \widetilde{a_1}+\widetilde{a_2}\le 2m_1-2$ we have
$$
n-r=\widetilde{a_1}+\widetilde{a_2} \quad \text{or} \quad n-r=\widetilde{a_1}+\widetilde{a_2}-m_1.
$$
If $n-r=\widetilde{a_1}+\widetilde{a_2}$, then we clearly have $n-r\equiv\widetilde{a_1}+\widetilde{a_2}\pmod{m_2}$. Otherwise, we have $n-r=\widetilde{a_1}+\widetilde{a_2}-m_1$. So, it can now be deduced that
$$
n+m_2=\widetilde{a_1}+r+\widetilde{a_2}+r,
$$
which is equivalent to say $n\equiv (\widetilde{a_1}+r)+(\widetilde{a_2}+r)\pmod{m_2}$.
\end{proof}

We now provide the proof of Theorem \ref{thm1}.	

\begin{proof}[Proof of Theorem \ref{thm1}]
Let $\varepsilon>0$ be an arbitrarily small given number. Then by Lemma \ref{lem2}, there is some prime $p$ so that
\begin{align}\label{eq2-1}
\sqrt{\frac{m}{4}}<p<\sqrt{\frac{m}{2(2-\varepsilon)}},
\end{align}
provided that $m$ is sufficiently large (in terms of $\varepsilon$). By Lemma \ref{lem4}, there is a subset $A\subset \mathbb{Z}_{2p^2}$ with $|A|\le 12p$ so that $\sigma_{A}(n)\ge 1$ for any $n\in \mathbb{Z}_{2p^2}$. From (\ref{eq2-1}), we know that
\begin{align}\label{eq2-2}
(2-\varepsilon)2p^2<m<2\times2p^2.
\end{align}
Thus, by Lemma \ref{lem5} there is a subset $B$ of $\mathbb{Z}_{m}$ with
\begin{align}\label{eq2-3}
|B|\le 2|A|\le 24p
\end{align}
such that $\sigma_B(n)\ge 1$ for any $n\in \mathbb{Z}_{m}$. Hence, by Lemma \ref{lem3} we have
$$
\ell_m=\min\limits_{\widetilde{A}\in \mathcal{H}_m}\left\lbrace m^{-1}\sum_{n\in \mathbb{Z}_m}\sigma_{\widetilde{A}}(n)\right\rbrace\le m^{-1}\sum_{n\in \mathbb{Z}_m}\sigma_{B}(n)=\frac{|B|^2}{m}.
$$
Employing (\ref{eq2-2}) and (\ref{eq2-3}), we get
$$
\frac{|B|^2}{m}\le \frac{(24p)^2}{(2-\varepsilon)2p^2}=144\times\frac{2}{2-\varepsilon}.
$$
Hence, it follows that
$$
\limsup_{m\rightarrow\infty}\ell_m\le 144\times\frac{2}{2-\varepsilon}
$$
for any $\varepsilon>0$, which clearly means that
$$
\limsup_{m\rightarrow\infty}\ell_m\le 144.
$$
This completes the proof of Theorem \ref{thm1}.
\end{proof}

\section{Proof of Theorem \ref{thm2}}
The proof of Theorem \ref{thm2} is based on the following remarkable result of Singer \cite{Singer}.

\begin{lemma}[Singer]\label{lem3-1}
Let $p$ be a prime. Then there exists a subset $A$ of $\mathbb{Z}_{p^{2}+p+1}$ so that $\delta_A(n)=1$ for any $n\in \mathbb{Z}_{p^{2}+p+1}$ with $n\neq \overline{0}$.
\end{lemma}

The next lemma is a variant of Lemma \ref{lem3}.
\begin{lemma}\label{lem3-2}
Let $m$ be a positive integer and $A$ a subset of $\mathbb{Z}_m$. Then
$$
\sum_{n\in \mathbb{Z}_m}\delta_A(n)=|A|^2,
$$
where $|A|$ denotes the number of elements of $A$.
\end{lemma}
\begin{proof}
It is clear that
$$
\sum\limits_{n\in \mathbb{Z}_m}\delta_A(n)=\sum\limits_{n\in\mathbb{Z}_m}\sum\limits_{\substack{a_1-a_2=n\\a_1,a_2\in A}}1=\sum\limits_{\substack{a_1,a_2\in A\\a_1-a_2\in \mathbb{Z}_m}}1=\sum\limits_{a_1,a_2\in A}1=|A|^{2}.
$$
This completes the proof of Lemma \ref{lem3-2}.
\end{proof}

We need another auxiliary lemma.
\begin{lemma}\label{lem3-3}
Let $\varepsilon>0$ be an arbitrarily small number. Let $m$ be a positive integer and $p$ a prime number with
$$
(2-\varepsilon)\big(p^{2}+p+1\big)<m<2\big(p^{2}+p+1\big).
$$
Suppose that $A$ is a subset of $\mathbb{Z}_{p^{2}+p+1}$ with $\delta_A(n)\ge 1$ for any $n\in \mathbb{Z}_{p^{2}+p+1}$, then there is a subset $B$ of $\mathbb{Z}_{m}$ with $|B|\le 2|A|$ such that $\delta_B(n)\ge 1$ for any $n\in \mathbb{Z}_{m}$.
\end{lemma}
\begin{proof}
Suppose that $m=\big(p^2+p+1\big)+r$, then $(1-\varepsilon)\big(p^2+p+1\big)<r< \big(p^2+p+1\big)$. Let
$$
B=A\cup \{a+r:a\in A\}.
$$
Then we have $|B|\le 2|A|$. It remains to prove $\delta_B(n)\ge 1$ for any $n\in \mathbb{Z}_{m}$.

Without loss of generality, we appoint that $0\le a\le p^2+p$ for any $a\in A$.
For $0\le n\le p^2+p$, there are two integers
$a_1,a_2\in A$ so that
$$
n\equiv a_1-a_2\pmod{p^2+p+1},
$$
which means that
$$
n= a_1-a_2 \quad \text{or} \quad n=a_1-a_2+\big(p^2+p+1\big)
$$
since $-p^2-p\le a_1-a_2\le p^2+p$.
If $n=a_1-a_2$, then we clearly have $n\equiv a_1-a_2\pmod{m}$. If $n=a_1-a_2+\big(p^2+p+1\big)$, then
$$
n-m=n-\big(p^2+p+1\big)-r=a_1-(a_2+r),
$$
from which it can be deduced that $n\equiv a_1-(a_2+r)\pmod{m}$.  In both cases we have
$\delta_B(n)\ge 1$ for any $0\le n\le p^2+p$. We are left over to consider the case $p^2+p+1\le n\le m-1$. In this case, we have
$$
0<n-r\le m-1-r=p^2+p.
$$
Thus, there are two elements $\widetilde{a_1},\widetilde{a_2}$ of $A$ so that
$$
n-r\equiv \widetilde{a_1}-\widetilde{a_2}\pmod{m}.
$$
Using again the constraint $-p^2-p\le \widetilde{a_1}-\widetilde{a_2}\le p^2+p$ we get
$$
n-r=\widetilde{a_1}-\widetilde{a_2} \quad \text{or} \quad n-r=\widetilde{a_1}-\widetilde{a_2}+\big(p^2+p+1\big).
$$
If $n-r=\widetilde{a_1}-\widetilde{a_2}$, then we clearly have $n-r\equiv \widetilde{a_1}-\widetilde{a_2}\pmod{m}$. Otherwise, we have $n-r=\widetilde{a_1}-\widetilde{a_2}+\big(p^2+p+1\big)$, from which it clearly follows
$$
n-m=\widetilde{a_1}-\widetilde{a_2}.
$$
So we also deduce $n\equiv \widetilde{a_1}-\widetilde{a_2}\pmod{m}$.
\end{proof}

We now turn to the proof of Theorem \ref{thm2}.
\begin{proof}[Proof of Theorem \ref{thm2}]
Let $\varepsilon>0$ be an arbitrarily small given number. Then by Lemma \ref{lem2}, there is some prime $p$ so that
\begin{align*}
\frac{\sqrt{2m-3}-1}{2}<p<\frac{\sqrt{\frac{4}{2-\varepsilon}m-3}-1}{2}
\end{align*}
providing that $m$ is sufficiently large (in terms of $\varepsilon$), which is equivalent to say
\begin{align}\label{eqthm3-1}
(2-\varepsilon)\big(p^{2}+p+1\big)<m<2\big(p^{2}+p+1\big).
\end{align}
By Lemma \ref{lem3-1}, there is a subset $A$ of $\mathbb{Z}_{p^{2}+p+1}$ so that $\delta_A(n)=1$ for any $n\in \mathbb{Z}_{p^{2}+p+1}$ with $n\neq \overline{0}$. Employing Lemma \ref{lem3-2}, we have
\begin{align*}
|A|^2=\sum\limits_{n\in \mathbb{Z}_{p^{2}+p+1}}\delta_A(n)=\sum\limits_{n\in \mathbb{Z}_{p^{2}+p+1},~n\neq \overline{0}}\delta_A(n)+\delta_A(0)=p^{2}+p+|A|,
\end{align*}
from which it follows clearly that
\begin{align*}
|A|=p+1.
\end{align*}
By Lemma \ref{lem3-3} and (\ref{eqthm3-1}), there is a subset $B$ of $\mathbb{Z}_{m}$ with
\begin{align}\label{eqthm3-2}
|B|\le 2|A|\le 2(p+1)
\end{align}
such that $\delta_B(n)\ge 1$ for any $n\in \mathbb{Z}_{m}$. Thus, by the definition of $g_m$ and Lemma \ref{lem3-2} again we have
$$
g_m=\min\limits_{\widetilde{A}\in \mathcal{K}_m}\left\lbrace m^{-1}\sum_{n\in \mathbb{Z}_m}\delta_{\widetilde{A}}(n)\right\rbrace\le m^{-1}\sum_{n\in \mathbb{Z}_m}\delta_{B}(n)=\frac{|B|^2}{m}.
$$
From (\ref{eqthm3-1}) and (\ref{eqthm3-2}), we get
$$
\frac{|B|^2}{m}\le \frac{4(p+1)^2}{(2-\varepsilon)\big(p^{2}+p+1\big)}\le \frac{4}{2-\varepsilon/2},
$$
provided that $m$ (hence $p$) is sufficiently large (in terms of $\varepsilon$). Hence, we conclude that
$$
\limsup_{m\rightarrow\infty}g_m\le \frac{4}{2-\varepsilon/2}
$$
for any $\varepsilon>0$, which clearly means that
$$
\limsup_{m\rightarrow\infty}g_m\le 2.
$$
This completes the proof of Theorem \ref{thm2}.
\end{proof}

\end{document}